\documentclass[11pt]{amsart}

\usepackage{amssymb,amsthm}

\pdfoutput=1

\newtheorem{theorem}{Theorem}

\newtheorem{lemma}[theorem]{Lemma}

\def\irr#1{{\rm  Irr}(#1)}

\def\ker#1{{\rm ker} (#1)}

\begin{document}

\title[Characters and Quotients]{${\rm B}_\pi$-characters and quotients}
\author[Mark L. Lewis]{Mark L. Lewis}

\address{Department of Mathematical Sciences, Kent State University, Kent, OH 44242}
\email{lewis@math.kent.edu}


\begin{abstract}
Let $\pi$ be a set of primes, and let $G$ be a finite $\pi$-separable group.  We consider the Isaacs ${\rm B}_\pi$-characters.  We show that if $N$ is a normal subgroup of $G$, then ${\rm B}_\pi (G/N) = \irr {G/N} \cap {\rm B}_\pi (G)$.
\end{abstract}

\maketitle

All groups in this paper are finite.  Let $\pi$ be a set of primes and let $G$ be a $\pi$-separable group.  In \cite{pisep}, Isaacs defined the subset ${\rm B}_\pi (G)$ of $\irr G$.  In this note we are going to prove the following:

\begin{theorem} \label{main}
Suppose $\pi$ is a set of primes and $G$ is a $\pi$-separable group.  If $N$ is a normal subgroup of $G$, then ${\rm B}_\pi (G/N) = \irr {G/N} \cap {\rm B}_\pi (G)$.
\end{theorem}

Following Gajendragadkar in \cite{gajen}, we say that a character $\chi \in \irr G$ is $\pi$-special if $\chi (1)$ is a $\pi$-number and for every subnormal subgroup $S$ of $G$, the irreducible constituents of $\chi_S$ have determinantal order that is a $\pi$-number.  In Proposition 7.1 of \cite{gajen}, Gajendragadkar proved that if $\alpha, \beta \in \irr G$ are characters so that $\alpha$ is $\pi$-special and $\beta$ is $\pi'$-special, then $\alpha \beta$ is irreducible, and this factorization is unique.  I.e., if $\alpha \beta = \alpha' \beta'$ where $\alpha'$ is $\pi$-special and $\beta'$ is $\pi'$-special, then $\alpha = \alpha'$ and $\beta = \beta'$.

Using \cite{pisep}, we say that $\chi \in \irr G$ is {\it $\pi$-factored} if there exists $\pi$-special $\alpha$ and $\pi'$-special $\beta$ so that $\chi = \alpha \beta$.  The following lemma regarding the kernels of $\pi$-factored characters is key to our argument.

\begin{lemma} \label{keylemma}
Suppose $\pi$ is a set of primes and $G$ is a $\pi$-separable group.  If $\chi \in \irr G$ satisfies $\chi = \alpha \beta$ where $\alpha$ is $\pi$-special and $\beta$ is $\pi'$-special, then $\ker {\chi} = \ker {\alpha} \cap \ker {\beta}$.
\end{lemma}

\begin{proof}
It is obvious that $\ker {\alpha} \cap \ker {\beta} \le \ker {\chi} = K$.  We need to show that $K \le \ker {\alpha} \cap \ker {\beta}$.  We first claim that $K \le Z (\alpha) \cap Z (\beta)$.  Suppose $g \in K$.  Then $\alpha (g)\beta (g) = \chi (g) = \chi (1) = \alpha (1) \beta (1)$.  Hence, $\alpha (1) \beta (1) = |\alpha (g) \beta (g)| = |\alpha (g)| |\beta (g)|$.  By Lemma 2.15(c) of \cite{text}, we know that $|\alpha (g)| \le \alpha (1)$ and $|\beta (g)| \le \beta (1)$.  The previous equality implies that these inequalities must be equalities, so $g \in Z (\alpha)$ and $g \in Z (\beta)$. This proves the claim.

By Lemma 2.27 (c) of \cite{text}, we see that $\alpha_K = \alpha (1) \mu$ and $\beta_K = \beta (1) \nu$ for linear characters $\mu$ and $\nu$ in $\irr K$.  Because $\alpha$ is $\pi$-special, $\mu$ must have $\pi$-order and because $\beta$ is $\pi'$-special, $\nu$ must have $\pi'$-order.  For $g \in K$, this implies that $\mu (g)$ is a $\pi$ root of unity and $\nu (g)$ is a $\pi'$-root of unity.
We have $\alpha(1) \beta (1) = \chi (1) = \chi (g) = \alpha (g) \beta (g) = \alpha (1) \mu (g) \beta (1) \nu (g)$.  This implies that $\nu (g) \mu (g) = 1$.  The only way that the product of a $\pi$-root of unity and a $\pi'$-root can equal $1$ is if they are both $1$.  I.e., we must have $\mu (g) = \nu (g) = 1$.  This implies that $\alpha (1) = \alpha (g)$ and $\beta (1) = \beta (g)$.  Therefore, $g \in \ker {\alpha} \cap \ker {\beta}$ as desired.
\end{proof}

We continue to let $\pi$ be a set of primes and $G$ be a $\pi$-separable group, and we fix $\chi \in \irr G$.  We say that $(S,\sigma)$ is a {\it subnormal pair} for $\chi$ if $S$ is a subnormal subgroup of $G$, $\sigma$ is an irreducible constituent of $\chi_S$.  In addition, we say that $(S,\sigma)$ is {\it  $\pi$-factored} if $\sigma$ is $\pi$-factored.  We can define a partial ordering on the subnormal pairs for $\chi$ by $(S,\sigma) \le (T,\tau)$ if $S \le T$ and $\sigma$ is a constituent of $\tau_S$.

Notice that $(1,1_1)$ is a $\pi$-factored subnormal pair for $\chi$, so there exists a maximal $\pi$-factored subnormal pair for $\chi$ with respect to the partial ordering.  It is shown in Theorem 3.2 of \cite{pisep} that the set of maximal $\pi$-factored subnormal pairs for $\chi$ are conjugate in $G$.  Let $(S,\sigma)$ be a maximal $\pi$-factored subnormal pair for $\chi$, and let $T$ be the stabilizer of $(S,\sigma)$ in $G$.  It is shown in Theorem 4.4 of \cite{pisep} that there is a unique $\tau \in \irr {T \mid \sigma}$ so that $\tau^G = \chi$.

We can now define the {\it $\pi$-nucleus} for $\chi$.  If $\chi$ is $\pi$-factored, then $(G,\chi)$ is the nucleus for $\chi$.  If $\chi$ is not $\pi$-factored, then let $(S,\sigma)$ be a maximal $\pi$-factored subnormal pair for $\chi$.  Let $T$ be the stabilizer of $(S,\sigma)$ in $G$, and let $\tau \in \irr {T \mid \sigma}$ so that $\tau^G = \chi$.  By Lemma 4.5 of \cite{pisep}, we know that $T < G$, so we can inductively define the $\pi$-nucleus of $\chi$ to be the $\pi$-nucleus of $\tau$.  Because the maximal $\pi$-factored subnormal pairs are all conjugate, it follows that the $\pi$-nucleus for $\chi$ is well-defined up to conjugacy.  (See the argument on page 108 of \cite{pisep}.)

If $(X,\eta)$ is a $\pi$-nucleus for $\chi$, then it is not difficult to see that $\eta$ must be $\pi$-factored.  As defined in Definition 5.1 of \cite{pisep}, we say that $\chi \in {\rm B}_\pi (G)$ if and only if $\eta$ is $\pi$-special where $(X,\eta)$ is a $\pi$-nucleus for $\chi$.

\begin{lemma}\label{one}
Let $\pi$ be a set of primes and let $G$ be a $\pi$-separable group.  Suppose that $N$ is a normal subgroup of $G$.  If $\chi \in \irr {G/N}$ has $\pi$-nucleus $(X,\eta)$, then $(X/N,\eta)$ is a $\pi$-nucleus for $\chi$ viewed as character in $\irr {G/N}$.
\end{lemma}

\begin{proof}
If $(X,\eta) = (G,\chi)$, then this is obvious.  Thus, we may assume that $X < G$.  Let $(S,\sigma)$ be a maximal $\pi$-factored subnormal pair for $\chi$ with stabilizer $T$ and character $\tau \in \irr {T \mid \sigma}$ so that $\tau^G = \chi$ and $(X,\eta)$ is a $\pi$-nucleus for $\tau$.  Notice that $(N,1_N)$ is a $\pi$-factored subnormal pair for $\chi$, so it is contained in a maximal such pair.  Since $N$ is normal, this implies that $N \le S$.  Because $\sigma$ is a constituent of $\chi_S$, we see that $N \le \ker {\sigma}$.  By Lemma \ref{keylemma}, we see that $\sigma$ is $\pi$-factored as a character in $\irr {S/N}$.  Notice that $(S/N,\sigma) \le (S^*/N,\sigma^*)$ if and only if $(S,\sigma) \le (S^*, \sigma^*)$, and by Lemma \ref{keylemma}, $\sigma^*$ is $\pi$-factored in $\irr {S^*/N}$ if and only if it is $\pi$-factored in $\irr {S^*}$.  Therefore, $(S/N,\sigma)$ must be a maximal $\pi$-factored subnormal pair for $\chi$ viewed as a character in $\irr {G/N}$.  It is immediate that $T/N$ will be the stabilizer for $(S/N,\sigma)$ in $G/N$ and that $\tau$ is the unique character in $\irr {T/N \mid \sigma}$ that induces $\chi$.  By induction, $(X/N,\eta)$ will the $\pi$-nucleus for $\tau$ viewed as a character of $\irr {T/N}$, and thus, $(X/N,\eta)$ will be the $\pi$-nucleus for $\chi$ viewed as a character of $\irr {G/N}$.
\end{proof}

We are now ready to prove Theorem \ref{main}.

\begin{proof}[Proof of Theorem \ref{main}]
Note that ${\rm B}_{\pi} (G/N) \subseteq \irr {G/N}$.  Hence, it suffices to show for $\chi \in \irr {G/N}$ that $\chi \in {\rm B}_\pi (G)$ if and only if $\chi \in {\rm B}_{\pi} (G/N)$.  Suppose $\chi \in \irr {G/N}$.  Let $(X,\eta)$ be a $\pi$-nucleus for $\chi$.  By Lemma \ref{one}, $(X/N,\eta)$ is a nucleus for $\chi$ viewed as a character of $\irr {G/N}$.  Note that $\eta$ is $\pi$-special as character in $\irr X$ if and only if it is $\pi$-special viewed as a character of $\irr {X/N}$.  We know that $\chi \in {\rm B}_{\pi} (G)$ if and only if $\eta$ is $\pi$-special character of $\irr X$ and $\chi \in {\rm B}_{\pi} (G/N)$ if and only if $\eta$ is $\pi$-special as a character of $\irr {X/N}$.  Since we saw that these are equivalent, this proves the theorem.
\end{proof}

\end{document}